\documentclass[reqno]{amsart}
\usepackage{amssymb,amsmath,amsthm}

\usepackage{setspace}
\numberwithin{equation}{section}

\newtheorem{thm}{Theorem}[section]

\theoremstyle{definition}

\theoremstyle{remark}
 \newtheorem{rmk}[thm]{Remark}

\newcommand{\R}{\mathbb{R}}

\newcommand{\Leb}{\mathcal{L}}
\newcommand{\n}{\mathcal{N}}

\newcommand{\spt}{\operatorname{spt}}
\newcommand{\Diff}{\operatorname{Diff}}

\begin{document}

\title[Optimal transport map]
{A strategy for non-strictly convex transport costs and the example of $||x-y||^p$ in  $\R^2$}

\author{Guillaume Carlier}\address{{\bf G.C.} CEREMADE, Universit\'e Paris Dauphine, Pl. de Lattre de Tassigny, 75775, Paris Cedex 16, FRANCE, \tt{carlier@ceremade.dauphine.fr},}

\author{Luigi De Pascale}\address{{\bf L.D.P.} Dipartimento di Matematica
  Applicata, Universit\'a di Pisa, Via Buonarroti 1/c, 56127
  Pisa, ITALY, \tt{depascal@dm.unipi.it},}

\author{Filippo Santambrogio}\address{{\bf F.S.} CEREMADE, Universit\'e Paris Dauphine, Pl. de Lattre de Tassigny, 75775, Paris Cedex 16, FRANCE, \tt{filippo@ceremade.dauphine.fr}.}

\begin{abstract} 
This paper deals with the existence of optimal transport maps for some optimal transport problems with a convex but non strictly convex cost. We  give a decomposition strategy to address this issue. As part of our strategy, we  have to treat some transport problems, of independent interest,  with a convex constraint on the displacement. As an illustration of our strategy, we prove  existence of optimal transport maps in the case where the source measure is absolutely continuous with respect to the Lebesgue measure and the transportation cost is of the form $h(\Vert x-y\Vert)$ with $h$ strictly convex increasing and $\Vert . \Vert$ an arbitrary norm in $\R^2$. 

\end{abstract}

\date{August 28 2009}
\maketitle

\section{Introduction}

Given two probability measures $\mu$ and $\nu$ on $\R^d$ and a
transport cost $c$ : $\R^d \to \R$ the corresponding Monge problem
consists in minimizing the average transport cost $\int_{\R^d}
c(x-T(x)) \mu(dx)$ among all \emph{transport maps} $T$ i.e. maps
pushing forward $\mu$ to $\nu$ (which as usual is denoted by
$T_\#\mu=\nu$). It is a highly nonconvex problem (whose admissible set
may even be empty if $\mu$ has atoms for instance) and it is therefore
relaxed to the Monge-Kantorovich problem that consists in minimizing
$\int_{\R^d\times \R^d} c(x-y) \gamma(dx,dy)$ over $\Pi(\mu, \nu)$, the set of
transport plans i.e. of probability measures on $\R^d\times \R^d$
having $\mu$ and $\nu$ as marginals. To prove existence of an optimal
transport map one thus aims to prove that there is an optimal plan
$\gamma$ (existence of such plans holds under very mild assumptions
since the Monge-Kantorovich problem is linear) that is in fact induced
by a transport map i.e. of the form $\gamma=({\rm{id}}, T)_\# \mu$. To
achieve this goal, one usually uses strongly the dual problem that
consists in maximizing $\int_{\R^d} \phi d \mu+\int_{\R^d} \psi d \nu$
subject to the constraint that $\phi(x)+\psi(y)\leq c(x-y)$. An
optimal pair $(\phi, \psi)$ for the dual is called a pair of
Kantorovich potentials. It is very well-known (under reasonable
assumptions on the measures) that when $c$ is strictly convex then
this strategy gives an optimal transport (see \cite{GaMc} or section
\ref{decomp} below where the arguments is briefly recalled; we also
refer to the book \cite{Vil} for a general overview and recent
developments of optimal transport theory).

It is also well-known that lack of strict convexity makes the
existence of an optimal transport much more delicate. Even the
important case (originally considered by Monge) where $c$ is a norm
was well understood only in recent years (\cite{EvaGan},
\cite{CafFelMcC}, \cite{AmbPra}, \cite{AmbKirPra}, \cite{ChaDeP},
\cite{Car}, \cite{ChaDeP09}, \cite{Car2}).  In the case of the
Euclidean norm, for example, the direction of the displacement is determined
by a Kantorovich potential and transport only takes place on a set of
segments called \emph{transport rays}. On the one hand, the lack of
strict convexity in the radial direction gives rise to an
indeterminacy of the displacement length, but on the other hand, the
problem on transport rays is one dimensional and then much
simpler. These observations lead to a strategy proof originally due to
Sudakov (\cite{Sud}) consisting in reducing to a one-dimensional
problem on each transport ray (monotone transport for instance) and
then glue the solutions together.  The most involved part of the full
proof consisted in proving that the transport rays have enough
regularity so that the disintegrated measures on such rays are non-atomic (see
\cite{CafFelMcC}, \cite{Car} for such a proof).

In the present paper we consider a convex but not strictly convex
$c$. We propose a decomposition strategy that takes advantage of the
fact that whenever the displacement is not fully determined it means
that it lies in some \emph{face} of $c$ but on such a face the cost is
affine and is therefore unchanged if we replace the transport plan by
another one which has the same marginals and satisfy the further
constraint that the displacement belongs to the face.  In the spirit,
our strategy can be compared to Sudakov's one but it is different
since there is no real analogue of transport rays here. Instead, our
strategy, detailed in section \ref{decomp}, involves "face restricted
problems" that are optimal transport problems with convex constraints
on the displacement. Such constrained problems have, we believe, their
own interest and motivation (e.g. due to connections with $L^{\infty}$
transportation problems as studied in \cite{ChaDepJuu}) and we will
address some of them in section \ref{constrained}. We will avoid here
subtle disintegration arguments to glue together the face restricted
problems in general but will instead illustrate in section \ref{cas2d}
how our strategy easily produces an optimal transport map in the case
of $c(z)=h(\Vert z\Vert)$ with $h$ strictly convex increasing and
$\Vert \cdot \Vert$ an arbitrary norm in $\R^2$. The contributions of this
paper are then :
\begin{itemize}
\item a general decomposition strategy to deal with convex but non
  strictly convex costs,
\item a contribution to constrained transport problems,
\item the proof of existence of optimal transport maps for a class of
  costs in $\R^2$.
\end{itemize}

\section{Strategy of decomposition}\label{decomp}
In this section we outline a general decomposition strategy to study
existence of an optimal transport for the Monge-Kantorovich problem
\begin{equation*}\label{e1wasse}
  \min\left\{\int_{\R^d\times\R^d}\!\!c(x-y)\,\gamma(dx,dy)\;:\;\gamma\in\Pi(\mu,\nu)\right\},
\end{equation*}
where $c$ is convex but not strictly convex.  Not all of the steps
listed below can be always carried on in full generality. In the
following sections we will detail some case where this is possible and
illustrate some applications.

The decomposition is based on the following steps:
\begin{itemize}
\item Consider an optimal plan $\gamma$ and look at the optimality
  conditions by means of a solution $(\phi,\psi)$ of the dual
  problem. From the fact that
$$\phi(x)+\psi(y)=c(x-y)\;\mbox{ on }\spt\gamma\quad\mbox{and}\quad 
\phi(x)+\psi(y)\leq c(x-y)$$ one deduces that if $x$ is a
differentiability point for $\phi$ (which is denoted $x\in
\Diff(\phi)$),
$$\nabla \phi(x)\in\partial c(x-y),$$
which is equivalent to
\begin{equation}\label{cstar}
  x-y\in \partial c^*(\nabla\phi(x)).
\end{equation}
Let us define
$$\mathcal{F}_c:=\{\partial c^*(p)\; : \; p\in\R^d\},$$
which is the set of all values of the subdifferential multi-map of
$c^*$. These values are those convex sets where the function $c$ is
affine, and they will be called {\it faces} of $c$.

Thanks to \eqref{cstar}, for every fixed $x$, all the points $y$ such
that $(x,y)$ belongs to the support of an optimal transport plan are
such that the difference $x-y$ belong to a same face of
$c$. Classically, when these faces are singleton (i.e. when $c^*$ is
differentiable, which is the same as $c$ being strictly convex), this
is the way to obtain a transport map, since only one $y$ is admitted
for every $x$.

Equation \eqref{cstar} also enables one to classify the points $x$ as
follows. For every $K\in \mathcal{F}_c$ we define the set
$$X_K:=\{x\in \Diff(\phi)\;:\; \partial c^*(\nabla\phi(x))=K\}.$$
Hence $\gamma$ may be decomposed into several subplans $\gamma_K$
according to the criterion $x\in X_K$, which, as we said, is
equivalent to $x-y\in K$.

If the set $\mathcal{F}_c$ is finite or countable, we can define
$$\gamma_K:=\gamma_{|X_K\times\R^d},$$
which is the simpler case. Actually, in this case, the marginal
$\mu_K:=(\pi_1)_\#\gamma_K$ (where $\pi_1(x,y)=x$) is a submeasure of
$\mu$, and in particular it inherits the absolute continuity from
$\mu$. This is often useful for proving existence of transport maps.

If $\mathcal{F}_c$ is uncountable, in some cases one can still rely on
a countable decompositions by considering the set
$\mathcal{F}_c^{multi}:=\{K\in \mathcal{F}_c\;:\;K\mbox{ is not a
  singleton }\}$. If $\mathcal{F}_c^{multi}$ is countable, then one
can separate those $x$ such that $\partial c^*(\nabla(\phi(x))$ is a
singleton (where a transport already exists) and look at a
decomposition for $K\in \mathcal{F}_c^{multi}$ only.

In some other cases, it could be useful to bundle together different
possible $K$'s so that the decomposition is countable, even if
coarser. We will give an example of this last type in Section 4.

\item This reduces the transport problem to a superposition of
  transport problems of the type
$$
\min \left\{\int_{\R^d\times\R^d}\!\!c(x-y)\,\gamma(dx,dy)\;:
\gamma\in\Pi(\mu_K,\nu_K),\; spt \gamma \subset \{x-y\in K\}\right\}.
$$
The advantage is that the cost $c$ restricted to $K$ is easier to
study.  For instance, if $K$ is a face of $c$, then $c$ is affine on
$K$ and in this case the transport cost does not depend any more on
the transport plan.
\item If $K$ is a face of $c$ the problem is reduced to find a
  transport map from $\mu_K$ to $\nu_K$ satisfying the constraint
  $x-T(x)\in K$, knowing a priori that a transport plan satisfying the
  same constraint exists.

  In some cases (for example if $K$ is a convex compact set with
  non-empty interior) this problem may be reduced to an $L^\infty$
  transport problem.  In fact if one denotes by $||\cdot||_K$ the
  (gauge-like) ``norm'' such that $K=\{x\,:\,||x||_K\leq 1\}$, one has
  \begin{equation}\label{linftyK}
    \min\Big\{\max\{||x-y||_K,\,(x,y)\in\spt(\gamma)\},\gamma \in\Pi(\mu,\nu)\Big\}\leq 1
  \end{equation}
  and the question is whether the same result would be true if one
  restricted the admissible set to transport maps only (passing from
  Kantorovich to Monge, say).  The answer would be positive if a
  solution of \eqref{linftyK} was induced by a transport map $T$
  (which is true if $\mu_K\ll\Leb^d$ and $K=\overline{B(0,1)}$, see
  \cite{ChaDepJuu}, but is not known in general). Moreover, the answer
  is also positive in $\R$ where the monotone transport solves all the
  $L^p$ problems, and hence the $L^\infty$ as well.
\item A positive answer may be also given in case (and it is actually
  almost equivalent) one is able to select, for instance by a
  secondary minimization, a particular transport plan satisfying
  $\spt(\gamma)\subset\{x-y\in K\}$ which is induced by a map.  This
  leads to the very natural question of solving
$$ \min\left\{\int_{\R^d\times\R^d}\!\![\frac 12 |x-y|^2+\chi_K(x-y)]\,d\gamma\;:\;\gamma\in\Pi(\mu,\nu)\right\},
$$
or, more generally, transport problems where the cost function
involves convex constraints on $x-y$. These problems are studied in
Section 3. For instance, in the quadratic case above, we can say that
the optimal transport $T$ exists and is given by
$T(x)=x-{\rm{proj}}_K(\nabla\phi(x))$ (${\rm{proj}}_K$ denoting
projection on $K$), provided two facts hold:
\begin{itemize}
\item $\mu$ is absolutely continuous, so that we can assure the
  existence of a (possibly approximate) gradient of $\phi$
  $\mu-$a.e. under mild regularity assumptions on $\phi$;
\item an optimal potential $\phi$ does actually exist, in a class of
  functions (Lipschitz, BV) which are differentiable almost
  everywhere, at least in a weak sense.
\end{itemize}
\item In order to apply the study of convex-constrained problems to
  the original problem with $c(x-y)$ the first issue (i.e. absolute
  continuity) does not pose any problem if the decomposition is finite
  or countable, while it is non trivial, and it presents the same kind
  of difficulties as in Sudakov's solution of the Monge's problem, in
  case of disintegration.  For interesting papers related to this kind
  of problems, see for instance \cite{CarDan,Car}.
\item As far as the second issue is concerned, this is much more
  delicate, since in general there are no existence results for
  potentials with non-finite costs.  In particular, a counterexample
  has been provided by Caravenna when $c(x,y)=|x-y|+\chi_K(x-y)$ where
  $K=\R^+\times\R^+\subset\R^2$ and it is easily adapted to the case
  of quadratic costs with convex constraints.  On the other hand, it
  is easy to think that the correct space to set the dual problem in
  Kantorovitch theory for this kind of costs would be $BV$
  since the the constraints on $x-y$ enable one to control the
  increments of the potentials $\phi$ and $\psi$ on some directions,
  thus giving some sort of monotonicity.  Yet, this is not sufficient
  to find a bound in $BV$ if an $L^\infty$ estimate is not available
  as well and the counterexample that we mentioned - which gives
  infinite values for both $\phi$ and $\psi$, exactly proves that this
  kind of estimates are hard to prove.
\end{itemize}

\begin{rmk}
  An interesting example that could be approached by this strategy is
  that of crystalline norms (a problem that has been already solved by
  a different method in \cite{AmbKirPra}). In this case the faces of
  the cost $c$ are polyhedral cones but, if the support of the two
  measures are bounded, we can suppose that they are compact convex
  polyhedra. This means, thanks to the considerations we did above,
  that it is possible to perform a finite decomposition and to reduce
  the problem to some $L^\infty$ minimizations for norms whose unit
  balls are polyhedra. In particular solving the $L^\infty$ problem
  for crystalline norms is enough to solve the usual $L^1$ optimal
  transport problem for the crystalline norms.
\end{rmk}

\begin{rmk}
  If, on the other hand, one wants a simple example that can be completely solved through this strategy and that works in any dimension, one can look at the cost $c(z)=(|z|-1)_+^2$, which vanishes for displacements smaller than one. In this case it is easy to see that $\mathcal{F}_c^{multi}$ has  only one element, which is given by the Euclidean ball $\overline{B(0,1)}$. The associated $L^\infty$ problem has been solved in \cite{ChaDepJuu} and this is enough to get the existence of an optimal transport.  
 \end{rmk}

\section{Constrained transport problems}\label{constrained}
In this section we see two useful examples of transport problem under
the constraint $x-y\in K$. Since the cost we use, due to this
constraint, is lower semicontinuous but not finitely valued, it is
well-known that duality holds (the minimum of the Kantorovich problem
coincides with the supremum of the dual one), but existence of
optimizers in the dual problem is not guaranteed. As we underlined in
Section 2, this is a key point and what we present here will always
assume (artificially) that this existence holds true. In section 4 
we will show a relevant example in which this is actually the case.

\subsection{Strictly convex costs with convex constraints}

We start from the easiest transport problem with convex constraints:

\begin{thm}
  Let $\mu,\,\nu$ be two probability measures on $\R^d$, with $\mu
  \ll\Leb^d$, $K$ a closed and convex subset of $\R^d$ and
  $h:\R^d\to\R$ a strictly convex function. Let $c(z)=h(z)+\chi_K(z)$:
  then the transport problem
$$\min\left\{\int_{\R^d\times \R^d} c(x-y)\,\gamma(dx,dy),\;\gamma\in\Pi(\mu,\nu)\right\}$$
admits a unique solution, which is induced by a transport map $T$,
provided that a transport plan with finite cost exists and the dual
problem admits a solution $(\phi,\psi)$ where $\phi$ is at least
approximately differentiable a.e.
\end{thm}

\begin{proof}
  As usual, consider an optimal transport plan $\gamma$ and the
  potentials $\phi$ and $\psi$ optimality conditions on optimal
  transport plans and optimal potentials for this problem read as
$$\phi(x)+\psi(y)=h(x-y)\;\mbox{ on }\spt\gamma\subset\{x-y\in K\}$$
and
$$ \phi(x)+\psi(y)\leq h(x-y)\;\mbox{ for all }(x,y)\,:\,x-y\in K$$
and, if $\phi$ is differentiable at $x$, they lead to
\begin{equation}\label{sumsubd}
  \nabla \phi(x)\in \partial c(x-y)=\partial h(x-y)+\n _K(x-y),
\end{equation}
where $N_K(z)$ is the normal cone to $K$ at $z$. We used the fact that
the sub-differential of the sum $h$ and $\chi_K$ is the sum of their
sub-differentials since $h$ is real-valued and hence
continuous. Equation \eqref{sumsubd} is verified by the true gradient
of $\phi$ if it exists but it stays true for the approximate gradient
if $\phi$ is only approximately differentiable.

Yet, when a vector $l$ and a point $\bar{z}\in K$ satisfy
$$l\in \partial h(\bar{z})+N_K(\bar{z}),$$
thanks to the convexity of $h$ and $K$, this gives that $\bar{z}$
minimizes $K\ni z\mapsto h(z)-l\cdot z$. Since $h$ is strictly convex
this gives the uniqueness of $\bar{z}$, which will depend on $l$. We
will denote it by $\bar{z}(l)$.

In this case, we get $x-y=\bar{z}(\nabla\phi(x))$, which is enough to
identify $y=T(x):=x-\bar{z}(\nabla\phi(x))$ and proving existence of a
transport map which is necessarily unique.
\end{proof}

\begin{rmk}
  Notice that, in the case $h(z)=\frac 12 |z|^2$, the point
  $\bar{z}(l)$ will be the projection of $l$ on $K$, which gives the
  nice formula
$$T(x)=x-{\rm{proj}}_K(\nabla\phi(x)),$$
i.e. a generalization of the usual formula for the optimal transport
in the quadratic case.
\end{rmk}

\subsection{Strictly convex costs of one variable and convex constraints}
Let $\mu\ll\Leb^2$ and $\nu$ be two probability measures in $\R^2$,
let $K$ be a convex subset of $\R^2$ such that
$\stackrel{\circ}{K}\neq \emptyset$ and consider the cost
$$ c(x-y)=h(x_1-y_1)+\chi_K (x-y),$$  
for a function $h:\R \to \R$ increasing and strictly convex.
\begin{thm}\label{main} Assume that there exist $\gamma \in
  \Pi(\mu,\nu)$ and $\phi, \psi $ Lipschitz such that
  \begin{equation}\label{poten1}
    \phi(x)+\psi (y) \leq c(x-y) \ \ \ \forall x,y 
  \end{equation}
  \begin{equation}\label{poten2}
    \phi(x)+\psi (y) = c(x-y) \ \ \ \gamma-a.e. 
  \end{equation}
  Then there exists an optimal transport map for the cost $c$ between
  $\mu$ and $\nu$
\end{thm}
\begin{rmk} The assumption above implies that $\gamma$ is an optimal
  plan and the couple $(\phi,\psi)$ is optimal for the dual
  problem. In order to have the existence of an optimal $\gamma$ it is
  enough to assume that there exists at least one $\sigma \in
  \Pi(\mu,\nu)$ such that
$$ \int_{\R^2\times \R^2}  c(x-y) \sigma(dx,dy) <\infty.$$
Again we recall that this assumption is not enough to have the
existence of an optimal couple for the dual problem.
\end{rmk}
\begin{proof} As $\Leb^2(\R^2 \setminus \Diff(\phi))=0$ also
  $\mu(\R^2 \setminus \Diff(\phi))=0 $ then we can restrict our
  attention to the points of differentiability of $\phi$. For each
  $(x,y)\in spt(\gamma)$ such that $x\in \Diff(\phi)$ by
  (\ref{poten1}) and (\ref{poten2}) we obtain
  \begin{equation}\label{subdiff}
    \nabla\phi(x)\in h'(x_1-y_1)e_1+ \n_K(x-y).
  \end{equation}
  By the convexity of the functions involved (\ref{subdiff}) is
  equivalent to
$$x-y \in arg\min_z  \{ h(z_1)-\nabla\phi(x)\cdot z +\chi_K (z)\}.$$
Let $S$ be the set of those $x$ such that the set on the right-hand
side of (\ref{subdiff}) is a singleton
$$arg\min_z  \{ h(z_1)-\nabla\phi(x)\cdot z +\chi_K (z)\}=\{p\},$$
then $y$ is uniquely determined by
$$y=x-p. $$
Let us consider a decomposition of $\gamma$ in two parts: $\gamma=
\gamma_{|S\times \R^2}+ \gamma_{|S^c \times \R^2}$ The first part of
the decomposition is already supported on the graph of a Borel map.
As far as the second part is concerned, we will prove the existence of
a transport map which gives the same cost and the same marginals. This
will be done by a sort of one-dimensional decomposition according to
the following observations.

Whenever the set on the right-hand side of (\ref{subdiff}) is not a
singleton then by the convexity of the function
$h(z_1)-\nabla\phi(x)\cdot z$ it is a convex subset of $K$.  Even
more, by the strict convexity of $h$ there exist a number $m(\nabla
\phi(x))$ such that
$$arg\min_z  \{ h(z)-\nabla\phi(x)\cdot z +\chi_K (z)\}\subset \{z\cdot e_1=m(\nabla \phi(x))\}\bigcap K .  $$
We claim that if $arg\min_z \{ h(z)-\nabla\phi(x)\cdot z +\chi_K (z)\}
$ has more than one element then $x$ is a local maximum of $\phi$ on
the line $x+te_2$.  In fact assume that there exist two points
$y=(m(\nabla \phi(x)),y_2)$ and $\tilde{y}=(m(\nabla \phi(x)),
\tilde{y}_2)$ such that $x-y \in arg\min_z \{ h(z)-\nabla\phi(x)\cdot
z +\chi_K (z)\} $ and $x-\tilde{y} \in arg\min_z \{
h(z)-\nabla\phi(x)\cdot z +\chi_K (z)\}$ and assume without loss of
generality that $\tilde{y}_2 < y_2$.  As $x-y \in K$ also $x+te_2 -y
\in K$ for small and positive $t$ (because it belongs to the segment
$[x-y,x-\tilde{y}]$), then
$$\phi (x+te_2)+\psi(y) \leq h(x_1-m(\nabla \phi(x))) $$
$$\phi(x)+\psi(y) =h(x_1-m(\nabla \phi(x))) $$
and subtracting the second equation from the first
$$ \phi(x+te_2)-\phi(x) \leq 0.$$
The same inequality is obtained for small, negative $t$ using
$\tilde{y}$.  Introduce now the set $M_n=\{x \ | \ \phi(x+te_2) \leq
\phi(x)\ \forall t \in [-\frac{1}{n},\frac{1}{n}]\}$.  The set $M_n$
is closed. Let $M_n^i:=M_n\cap \{x \ : \ x\cdot e_2 \in [\frac{i}{n},
\frac{i+1}{n}]\} $, as a consequence of local maximality, $\phi$ is
vertically constant on $M_n^i$.  There exists a function
$\tilde{\phi}$ depending only on the variable $x_1$ such that
$\tilde{\phi}$ coincide with $\phi$ on $M_n^i$. Then on the set $\Leb
(M_n^i)\cap \Diff(\phi)$, $\nabla \phi=app\nabla \phi=app\nabla
\tilde{\phi}$ and the latter approximate gradient depends only on
$x_2$ which implies that $\nabla \phi$ is vertically constant on a
subset of full measure of $M_n^i$.  The disjoint union $S^c
=\cup_{n,i} [S^c \cap (M_n^i\setminus M_{n-1})] $ induces the
decomposition
$$\gamma_{|S^c \times \R^2}= \sum_{n,i}\gamma_{n,i},$$
where $\gamma_{n,i}:=\gamma_{|[S^c \cap (M_n^i\setminus
  M_{n-1})]\times \R^2}$.

Denote by $\mu_{n,i}$ and $\nu_{n,i}$ the marginals of $\gamma_{n,i}$.
Clearly $\mu_{n,i}$ is absolutely continuous with respect to
$\Leb^2$. Consider the disintegration according to $x_1$
$$\mu_{n,i}= a_{n,i}(x_1)\cdot\Leb^1 \otimes \mu_{n,i}^{x_1} $$
with $\mu_{n,i}^{x_1}\ll\Leb^1$ for almost every $x_1$, where
$a_{n,i}(x_1)\cdot\Leb^1$ is the projection of $\mu_{n,i}$ on the
first variable (which is absolutely continuous as well). Analogously,
we define $\nu_{n,i}^{y_1}$ as the disintegration of $\nu_{n,i}$ with
respect to $y_1$. From what we said, it follows that there exist a
map $T_{n,i}:\R \to \R$ such that $spt \gamma_{n,i} \cap \{x \ : \ x
\cdot e_1 =x_1\} \subset \{(x,y)\ : \ y\cdot e_1=T_{n,i}(x_1)
\}$. Notice that $T_{n,i}(x_1)$ coincides with $m(\nabla\phi(x))$ for
every $x\in M_n^i$ whose first component is $x_1$.

Also disintegrate
$$\gamma_{n,i}=a_{n,i}(x_1)\cdot\Leb^1 \otimes \gamma_{n,i}^{x_1} $$
with $ \gamma_{n,i}^{x_1}\in \Pi( \mu_{n,i}^{x_1}, \delta
_{T_{n,i}(x_1)} \otimes \nu_{n,i}^{T_{n,i} (x_1)} ). $ It is then
enough to replace $ \gamma_{n,i}$ with a new transport plan of the
form $({\rm{id}}\times S_{n,i})_\#\mu_{n,i}$ with the following
properties
\begin{gather*}
  S_{n,i}(x_1,x_2)=(T_{n,i}(x_1),S_{n,i}^{x_1}(x_2)),\\
  (S_{n,i}^{x_1})_\#\mu_{n,i}^{x_1}=\nu_{n,i}^{T_{n,i}(x_1)},\quad
  x_2-S_{n,i}^{x_1}(x_2)\in K_{x_1-T_{n,i}(x_1)},
\end{gather*}
where we denote, for every $t$, the section of $K$ at level $t$ by
$K_t:=\{s\;:\;(t,s)\in K\}$.
 
Since the $h$-part of the cost only depends on $x_1-y_1$ and
$\gamma_{n,i}$ and the transport plan issued by the map $S_{n,i}$
realize the same displacements as far as the first coordinates are
concerned, this part of the cost does increase. Moreover, the
constraint $(x_1-y_1,x_2-y_2) \in K$, which amounts to the requirement
that $x_2-y_2$ belongs to a segment $K_{x_1-y_1}$, is preserved,
thanks to the last condition above. Hence, the two transport plans
have exactly the same cost.

We are only left to find, for every $x_1$, the map
$S_{n,i}^{x_1}$. For this, it is enough to choose the monotone
transport map from $\mu_{n,i}^{x_1}$ to $\nu_{n,i}^{x_1}$. This map is
well defined because the first measure is absolutely continuous (in
particular there are no atoms). The constraint is preserved because
$\gamma_{n,i}^{x_1}$ satisfied it, and the monotone transport map is
optimal for any convex transport cost of $x_2-y_2$ .

Notice that the whole map $S_{n,i}$ we are using is measurable since
in every ambiguous case we chose the monotone transport map.
\end{proof}

\section{Application: $c(x-y)=h(\|x-y\|)$ in $\R^2$}\label{cas2d}
\begin{thm}
  Let $\mu,\,\nu$ be probability measures compactly supported in
  $\R^2$, with $\mu\ll\Leb^2$, let $||\cdot||$ be an arbitrary norm of
  $\R^2$ and $h:\R^+\to\R^+$ a strictly convex and increasing
  function. Denote by $c$ the function $c(z)=h(||z||)$, which is still
  a convex function. Then the transport problem
$$\min\left\{\int_{\R^d\times \R^d} c(x-y)\; \gamma(dx,dy),\;\gamma\in\Pi(\mu,\nu)\right\}$$
admits at least a solution which is induced by a transport map $T$.
\end{thm}
\begin{proof}
  The strategy to prove this theorem is almost the one described in
  Section 2, with the additional trick of bundling together some faces
  of the convex function $h(||x-y||)$. The boundary of the unit ball
  of the norm $||\cdot||$ has a countable number of flat parts (they
  are countable since each one has a positive length, otherwise it
  would not be a flat part, and we cannot have more than a countable
  quantity of disjoint positive length segments in the boundary). We
  call $F_i$, with $i=1,2,\dots$ the closure of these flat parts and
  we then associate to each face $F_i$ the cone $K_i=\{tz,\;t\geq
  0,\;z\in F_i\}$.

  Consider an optimal transport plan $\gamma$ and a pair of
  Kantorovich potentials $(\phi,\psi)$ (these objects exist since the
  cost is continuous). From the general theory of optimal transport
  and what we underlined in Section 2, it follows that if
  $(x,y)\in\spt(\gamma)$ and $x$ is a differentiability point of
  $\phi$, then $\nabla\phi(x)\in\partial c(x-y)$. This may be
  re-written as
$$x-y\in\partial c^*(\nabla\phi(x)).$$
We classify the points $x$ according to $\partial
c^*(\nabla\phi(x))$. This subdifferential may be either a singleton or
a segment (it cannot be a set of dimension two because otherwise $c$
would be affine on a set with non-empty interior, which is in
contradiction with the strict convexity of $h$). We call $X_0$ the set
of those $x$ such that $\partial c^*(\nabla\phi(x))$ is a singleton.
If instead it is a segment, it is a segment of points which share a
single vector in the subdifferential of $c$. This means, thanks to the
shape of $c$, that it is necessarily an homothety of one face $F_i$:
$$x-y\in\partial c^*(\nabla\phi(x))=t(x)F_{i(x)}\subset K_{i(x)}.$$
We denote by $X_i$ the set of points $x$ where $c^*(\nabla\phi(x))$ is
an homothety of $F_{i}$.

The disjoint decomposition given by the sets $X_i$ induces a
corresponding decomposition $\gamma=\sum_{i\geq 0} \gamma_i$ where
$\gamma_i=\gamma_{|X_i\times\R^2}$.  Each sub-transport plan
$\gamma_i$ is an optimal transport plan between its marginals $\mu_i$
and $\nu_i$, which are submeasures of $\mu$ and $\nu$,
respectively. In particular, each measure $\mu_i$ is absolutely
continuous.

We now build a solution to the optimal transport problem for the cost
$c$ that is induced by a transport map by modifying $\gamma$ in the
following way: $\gamma_0$ is already induced by a transport map since
for every $x$ we only have one $y$; every $\gamma_i$ for $i\geq 1$
will be replaced by a new transport plan with the same marginals which
does not increase the cost, thanks to Theorem \ref{main}. Since
$\gamma_i$ sends $\mu_i$ to $\nu_i$, is optimal for the cost $c$ and
satisfies the additional condition $\spt(\gamma_i)\subset
\{(x,y)\,:\,x-y\in K_i\}$, it is also optimal for the cost
$c(x-y)+\chi_{K_i}(x-y)$. The advantage of this new cost is that $c$
depends on one variable only, when restricted to $K_i$ since there
exists a new basis $(e_1,e_2)$ such that $||z||=z_1$ for every $z\in
K_i$ which is written as $(z_1,z_2)$ in this new basis (the direction
of the vector $e_1$ being the normal direction to $F_i$). Theorem
\ref{main} precisely states that the optimal transport problem for a
cost which is given by a strictly convex function of one variable plus
a constraint imposing that $x-y$ belongs to a given convex set admits
a solution induced by a transport map, provided the first measure is
absolutely continuous and some $\gamma$, $\phi$ and $\psi$ satisfying
equations \ref{poten1} and \ref{poten2} exist. This last assumption is
exactly satisfied by taking $\gamma_i$ and the original potentials
$\phi$ and $\psi$ for the problem with cost $c$ from $\mu$ to $\nu$,
since they satisfied
\begin{gather*}
  \phi(x)+\psi(y)\leq h(||x-y||)\leq h(||x-y||)+\chi_{K_i}(x-y),\\
  \phi(x)+\psi(y)= h(||x-y||)= h(||x-y||)+\chi_{K_i}(x-y)\mbox{ on
  }\spt(\gamma_{K_i}).\qedhere
\end{gather*}
\end{proof}

\section*{Acknowledgments}
The work of the authors has been partially supported by the University Paris Dauphine (CEREMADE) and the University of Pisa through the project 
"Optimal Transportation and related topics". G.C. and F.S. gratefully acknowledge the support of the Agence Nationale de la Recherche through the project  "ANR-07-BLAN-0235 OTARIE". 

The authors wish to thank  the Centro de Ciencias de Benasque Pedro Pascual (CCBPP)
where this work was completed during the program ``Partial differential equations, optimal design and numerics'', August 23 - September 4 2009.



\end{document}